\documentclass[12pt]{amsart}
\usepackage{amssymb}
\usepackage{amsmath}
\usepackage{amsthm}
\usepackage{amsfonts}
\usepackage{enumerate}
\usepackage{url}

\makeatletter
\@namedef{subjclassname@2010}{%
  \textup{2010} Mathematics Subject Classification}
\makeatother

\newtheorem{theorem}{Theorem}[section]
\newtheorem{conjecture}[theorem]{Conjecture}
\newtheorem{corollary}[theorem]{Corollary}
\newtheorem{lemma}[theorem]{Lemma}
\newtheorem{hypothesis}[theorem]{Hypothesis}

\theoremstyle{definition}

\newtheorem*{xrem}{Remark}

\begin{document}

\baselineskip=17pt

\title[Lower bound hypothesis]{The $3x+1$ problem: a lower bound hypothesis}

\author[O. Rozier]{Olivier Rozier}
\address{237, rue de Romainville\\
93230 Romainville, France}
\email{olivier.rozier@gmail.com}

\date{}

\begin{abstract}
Much work has been done attempting to understand the dynamic behaviour of the so-called ``$3x+1$" function. It is known that finite sequences of iterations with a given length and a given number of odd terms have some combinatorial properties modulo powers of two. In this paper, we formulate a new hypothesis asserting that the first terms of those sequences have a lower bound which depends on the binary entropy of the ``ones-ratio". It is in agreement with all computations so far. Furthermore it implies accurate upper bounds for the total stopping time and the maximum excursion of an integer. Theses results are consistent with two previous stochastic models of the $3x+1$ problem.
\end{abstract}



\maketitle

\section{Introduction}
\label{sec:intro}
Let us consider the $T$ function acting on the set of positive integers and defined by
\begin{equation}
T(n) = \begin{cases}
\frac{3n+1}2 &\text{if $n$ is odd,}\\
\frac n2 &\text{otherwise.}
\end{cases}
\end{equation}
It is expected but not yet proved that, whatever the initial value of $n$, the repeated iterations of $T$ reach the value 1 at some point, thus entering the infinite loop $1, 2, 1, 2, \ldots$ called the \emph{trivial cycle}. This question is notoriously intractable, despite its simple statement, and has received various names like the $3x+1$ \emph{problem}, the \emph{Syracuse problem} or the \emph{Collatz conjecture} \cite{Lag10}.

\begin{conjecture}{\em ($3x+1$ problem)}\label{conj:T} 
For any integer $n > 0$, we have $T^{(j)}(n) = 1$ for some $j \geq 0$,
where $T^{(j)}$ denotes the $j$-th iterate of $T$.
\end{conjecture}

The $3x+1$ problem may be divided into Conjectures \ref{conj:DT} and \ref{conj:NC} below, asserting the absence of any other dynamic than the trivial cycle.

\begin{conjecture}{\em (Absence of divergent trajectory)}\label{conj:DT}
For all positive integer $n$, the infinite sequence
 $\left(   T^{(k)}(n) \right)  _{k=0}^{\infty}$, called the trajectory of $n$, is bounded. 

\end{conjecture}

\begin{conjecture}{\em (Absence of non-trivial cycle)}\label{conj:NC}
There exist no integers $n > 2$ and $j > 0$ such that $T^{(j)}(n) = n$.
\end{conjecture}

We propose a heuristic approach, which is greatly inspired by a well-known paper of Lagarias \cite{Lag85}, and we mostly follow his notations and denominations throughout the paper.

Combinatorial properties of $T$ iterations are leading us to formulate a new hypothesis (see  \S\ref{subsec:hypothesis})  involving the binary entropy function. So far, this function rarely appears in the vast literature on the $3x+1$ problem with a few notable exceptions (e.g., \cite[p. 84]{Wir98}). It has been used by Lagarias \cite{Lag85} to estimate the density of integers whose stopping time is bounded by a given value, thus improving a previous result of Terras \cite{Ter76}. Tao made a similar calculation to give a heuristic estimation of the number of non-trivial cycles, concluding that very likely there is none \cite{Tao11}. Let us also mention the application by Sinai of the notion of entropy of a dynamical system within a statistical modelling of the $3x+1$ problem \cite{Kon10, Sin03}. Besides, the binary entropy function is widely used in the context of information theory to express the entropy of Bernoulli processes.

In \S\ref{sec:syracuse}, we will see that proving our hypothesis would be more than sufficient to solve the $3x+1$ problem. Unexpectedly, it would further imply accurate upper bounds for the total stopping time and maximum excursion, which constitute the main result of the present paper (cf. Theorem \ref{th:dyn}). A brief comparison will be carried out with the predictions of the random walk model \cite{Lag92}. Then, in \S\ref{sec:model}, we analyze a simple random model that is supporting our hypothesis.

Finally, we investigate a particular case of our hypothesis related to finite sequences of $T$ iterations with only one even term.

\section{A lower bound hypothesis}
\label{sec:hypothesis}

\subsection{Combinatorial and heuristic approach}
\label{subsec:heuristic}

The $T$ function exhibits remarkable combinatorial properties under iterations. Indeed, if we consider for each positive integer $n$ and length $j$ the \textit{parity vector}
\begin{equation}
V_{j}(n) = \left(  n, T(n), \ldots , T^{(j-1)}(n) \right)   \mod 2,
\end{equation}
then we have the following result proved independently by Terras \cite{Ter76} and Everett \cite{Eve77}:
\begin{theorem}{\em (Terras)} \label{th:terras}
Two positive integers $n$ and $m$ have same parity vector of length $j$ if and only if $n \equiv m \pmod{ 2^{j}}$.
\end{theorem}

An immediate consequence is that every positive integer $n \leq 2^{j}$ is uniquely identified by its parity vector $V_{j}(n)$. Hereafter, let $I(j,q)$ denote the set of positive integers $n$ for which there are exactly $q$ occurrences of 1 in $V_{j}(n)$. Then, from $I(j,q)$, we extract the finite subset
\begin{equation}
I_{0}(j,q) = I(j,q) \cap \lbrace 1, 2, 3, \ldots, 2^{j} \rbrace.
\end{equation}
Conversely, it follows from Theorem \ref{th:terras} that $I(j,q)$ is the set of congruence classes modulo $2^j$ of $I_{0}(j,q)$  over the positive integers. It is easily seen that, for any fixed  $j$, the set $\left\lbrace  I_{0}(j,0), \ldots, I_{0}(j,j) \right\rbrace $ is a partition of $\lbrace 1, \ldots, 2^{j} \rbrace$ such that
\begin{equation}\label{eq:card}
\# I_{0}(j,q) = \binom{j}{q} \quad \text{for }q = 0, \ldots, j
\end{equation}
where $\#$ denotes the cardinality. As an example, we exhibit for $j=6$ the partition of $\lbrace 1, \ldots, 64 \rbrace$:
\begin{align*}
I_{0}(6,0) &= \lbrace 64 \rbrace, \\
I_{0}(6,1) &= \lbrace 16, 20, 21, 32, 40, 42 \rbrace, \\
I_{0}(6,2) &= \lbrace 4, 5, 6, 8, 10, 12, 13, 24, 26, 34, 35, 48, 49, 52, 53 \rbrace, \\
I_{0}(6,3) &= \lbrace 1, 2, 3, 11, 17, 22, 23, 25, 28, 29, 36, 37, 38, 44, 45, 46, 50, 51, 56, 58 \rbrace, \\
I_{0}(6,4) &= \lbrace 7, 9, 14, 15, 18, 19, 30, 33, 43, 54, 55, 57, 59, 60, 61 \rbrace, \\
I_{0}(6,5) &= \lbrace 27, 31, 39, 41, 47, 62 \rbrace, \\
I_{0}(6,6) &= \lbrace 63 \rbrace.
\end{align*}

Lagarias suggested in \cite{Lag85} that the $T$ function has some mixing properties under iteration, modulo powers of 2. One may further verify that the cumulative distribution function of $I_{0}(j,q)$ from 1 to $2^{j}$ appears fairly linear, for large $j$ and $q$ values. Therefore we may expect that the distribution of $I_{0}(j,q)$ over $[0, 2^{j}]$ tends to be uniform and, roughly, that 
\begin{equation}\label{eq:approx_min}
\min I_{0}(j,q) \approx \frac{2^j}{\binom{j}{q}}.
\end{equation}
Also, one may wonder whether a lower bound of the form
\begin{equation}\label{eq:min0}
\min I_{0}(j,q) \geq j^{-C} \frac{2^j}{\binom{j}{q}}.
\end{equation}
holds for some positive constant $C$.

\subsection{Hypothesis}
\label{subsec:hypothesis}

The previous heuristic approach leads us to formulate the hypothesis to which this paper is dedicated:

\begin{hypothesis} \label{hyp:LBH1} 
For each $j \geq 1$ and $0 \leq q \leq j$, let $I_{0}(j,q)$ be the set of integers $n$ with $1 \leq n \leq 2^j$ such that the vector
$$ \left(  n, T(n), T(T(n)), \ldots, T^{(j-1)}(n) \right) $$
contains exactly $q$ odd integers. Then there is a real constant $C \geq 0$ such that, for all $j \geq 2$ and all $0 \leq q \leq j$, the set $I_{0}(j,q)$ is bounded from below by
\begin{equation}\label{eq:min1}
 j^{-C} \cdot 2^{\left( 1-H(r) \right) j}
\end{equation}
where $r=q/j$ is called the ``ones-ratio" and $H$ is the binary entropy function $H(x) = -x \log_{2}(x) - (1-x) \log_{2}(1-x)$.
\end{hypothesis}

In the literature on the $3x+1$ problem, the ``ones-ratio" usually denotes the proportion of odd terms in a sequence leading to the value 1. Hereafter, we extend this notion to all finite sequences of iterations, whatever the value of the last term.

The introduction of the binary entropy function $H$ in Hypothesis \ref{hyp:LBH1} is due to the first order approximation 
\begin{equation}\label{eq:binom_H}
 \log_2 \binom jq \sim  H\left(\frac qj \right) j
\end{equation}
for large values of $j$ and $q$, which can be derived from the Stirling formula. Let us recall that $H$ is a concave function with a single maximum $H(1/2) = 1$ and two minima $H(0) = H(1) = 0$ by continuous extension.

The value $H(r)$ in \eqref{eq:min1} is a measure of the entropy in the set of parity vectors $V_j(n)$ for $n \in I_{0}(j,q)$.

Here we provide another formulation of Hypothesis \ref{hyp:LBH1} with slightly less restrictive conditions, as it may be applied to the infinite sets $I(j,q)$ already introduced in $\S\ref{subsec:heuristic}$, and includes the case $j=1$.

\begin{hypothesis} \label{hyp:LBH2} $\textbf{\emph{(Lower Bound Hypothesis - LBH)}}$
There is a real constant $C \geq 0$ such that for all positive integers $j$ and $n$ not both equal to 1, we have
\begin{equation} \label{eq:min} 
 n \geq j^{-C} \cdot 2^{\left( 1-H\left( \frac{q}{j}\right)  \right) j}
\end{equation}
where $q$ is the number of odd integers in the vector $\left(  n, T(n), \ldots, T^{(j-1)}(n) \right)$.
\end{hypothesis}

 It is easy to see that Hypotheses \ref{hyp:LBH1} and \ref{hyp:LBH2} are  equivalent with the same constant $C$. For convenience purposes, we shall simply refer to both of them by LBH. Nonetheless, the formulation from Hypothesis \ref{hyp:LBH2} will prove to be more suitable when studying all the implications related to the $3x+1$ problem.

One may first verify that LBH holds and is quite sharp in many cases, in the sense that there is an integer $n \in I_{0}(j,q)$ for which the lower bound is reached with a small value of $C$. For example, it is sharp with $C$ near zero for the two extremal values of the ones-ratio $r=0$ and $r=1$, since $I_{0}(j,0)=\{ 2^j \}$  and $I_{0}(j,j) = \{ 2^j - 1 \}$ for all positive integer $j$. Remarkably, it is also sharp with $C=0$ in the central case $r=1/2$ arising when $j=2q$. Indeed the set $I_{0}(2q,q)$ contains the integer 1 which has parity vector $V_{j}(1) = (  1, 0, 1, 0, \ldots )$.

In fact, LBH is true for any $C \geq 0$ in all cases where $r=q/j \leq 1/2$. This is a consequence of Lemma \ref{lem:hyp_ok0} below together with the inequality
\begin{equation} 
H(x) \geq 2x, \quad \text{for }x \leq \frac 12
\end{equation}
which follows from the concavity of the $H$ function.

\begin{lemma} \label{lem:hyp_ok0}
For all $0 \leq 2q \leq j$,
$ \min I_{0}(j,q) = 2^{j - 2q}.$
\end{lemma}

\begin{proof}
For all $n \in I_{0}(j,q)$, we write 
\[ 1 \leq T^{(j)}(n) \leq 2^{2q - j} n\]
where the second inequality is easily obtained from the fact that $T(m) \leq 2m$ for any  integer $m$.
It follows that
\[ \min I_{0}(j,q) \geq 2^{j - 2q}.\]
To complete the proof, observe that $2^{j - 2q} \in I_{0}(j,q)$.
\end{proof}

As a result of Lemma \ref{lem:hyp_ok0} along with the strict concavity of the $H$ function, it turns out that the lower bound in LBH is not sharp when the ones-ratio is strictly between 0 and 1/2. However, the numerical results in \S\ref{subsec:comp} will show that it can be sharp for sequences that tend to grow.

The forthcoming Theorem \ref{th:hyp_ok} states the validity of the inequality \eqref{eq:min} for a large part of the sequences with a ones-ratio between  1/2 and $r_H = 0.609\ldots$\footnote{The value of $r_H$ already appears in various papers of Lagarias and coauthors (e.g., \cite[p. 140]{Kon10}) as the ones-ratio upper limit for finite sequences leading to 1 in stochastic models of the $3x+1$ problem.}. It relies on Lemma \ref{lem:ratio}, which is a generalization of a formula by Eliahou \cite{Eli93} regarding all cycles of the $T$ function.

\begin{lemma}  \label{lem:ratio}
Let $1 \leq q \leq j$ and $n \in I(j,q)$. Then
\begin{equation} \label{eq:ratio}
\frac{T^{(j)}(n)}{n} = 2^{-j}\prod_{k=0}^{q-1} \left( 3 + m_{k}^{-1} \right)
\end{equation}
where $m_{0}, \ldots, m_{q-1}$ are the odd terms in the finite sequence $\left(   T^{(k)}(n) \right)  _{k=0}^{j-1}$.
\end{lemma}

\begin{proof}
This result is straightforward to prove by applying the same method as in \cite{Eli93}. Indeed, we have 

\[\frac{T^{(j)}(n)}{n} = \prod_{k=0}^{j-1} \frac{T^{(k+1)}(n)}{T^{(k)}(n)} = \frac{1}{2^{j-q}} \prod_{k=0}^{q-1} \left( \frac{3 + m_{k}^{-1} }{2} \right)\]
since $j-q$ is the number of even terms among $n, T(n), \ldots, T^{(j-1)}(n)$.
\end{proof}

\begin{theorem}
\label{th:hyp_ok}
Let $1 \leq q < j$ such that $r = q/j \leq \rho^{-1} = 0.630\ldots$, where $\rho = \log_{2}3$, and let $n \in I(j,q)$ for which the terms $n, T(n), \ldots, T^{(j)}(n)$ are all distinct. Then
\begin{equation} \label{eq:sharp}
n \geq j^{-\frac 16} \cdot 2^{(1-\rho r)j}.
\end{equation}
Assume further that $r \leq r_H = 0.609089767\ldots$, where $r_{H}$ is the unique non-zero real number such that
\begin{equation} \label{eq:rH}
H(r_{H}) \log 2 = r_{H} \log 3.
\end{equation}
Then there holds the lower bound
\begin{equation}
 n \geq j^{-\frac 16} \cdot 2^{\left( 1-H(r) \right) j}.
\end{equation}
\end{theorem}

\begin{proof}
From Lemma \ref{lem:ratio}, we write
\[ \frac 1n \leq \frac{T^{(j)}(n)}{n} = 2^{-j}\prod_{k=0}^{q-1} \left( 3 + m_{k}^{-1} \right)\]
where $m_{0}, \ldots, m_{q-1}$ are the odd terms among $n, T(n), \ldots, T^{(j-1)}(n)$. This gives
\begin{equation} \label{eq:lower_bound_logn}
 \log n \geq j \log2 - q \log3 - \sum_{k=0}^{q-1} \log \left(  1 + \frac{1}{3 m_{k}}\right)
\end{equation}
with 
\begin{equation}
\sum_{k=0}^{q-1} \log \left(  1 + \frac{1}{3 m_{k}}\right) \leq \frac 13 \sum_{k=0}^{q-1}  \frac{1}{m_{k}} 
\end{equation}
by applying $\log(1+x) \leq x$. Since $m_{0}, \ldots, m_{q-1}$ are distinct odd numbers strictly greater than 1, we get
\begin{equation} \sum_{k=0}^{q-1}  \frac{1}{m_{k}} \leq \sum_{k=1}^{q}  \frac{1}{2k+1}.
\end{equation}
Then, from the upper bound 
\[\sum_{k=1}^{q}  \frac{1}{2k+1} \leq \frac 12 \log q + \log 2 + \frac 12 \gamma_{euler} -1 + \frac{1}{2q}\]
where $\log 2 + \frac 12 \gamma_{euler} = 0.981\ldots$, we infer
\begin{equation} \label{eq:harmonic}
 \sum_{k=1}^{q}  \frac{1}{2k+1} \leq \frac 12 \log q - \frac 12 \log r = \frac 12 \log j
\end{equation}
for $q \geq 3$, using the fact that $r \leq \rho^{-1}$. It is easy to state inequality \eqref{eq:harmonic} for $q=1,2$.  E.g., for $q=1$, we check that
\[ \frac 13 < \frac 12 \log 2 \leq \frac 12 \log j .\]
It follows from the inequalities \eqref{eq:lower_bound_logn} to \eqref{eq:harmonic} that
\[ \log n \geq j \log2 - q \log3 -\frac 16 \log j,\]
or equivalently,
\[ n \geq j^{-\frac 16} 2^{(1-\rho r)j}.\]
To conclude the proof, observe that the $H$ function is strictly concave. This implies that there is a unique non-zero real $r_H$ for which $H(r_H) = \rho r_H$, and that $H(x) \geq \rho x$ if and only if $x \leq r_H$. As a result, if we further assume that $r \leq r_H$, it yields $\rho r \leq H(r)$ and
\[ n \geq j^{-\frac 16} 2^{\left( 1-H(r)\right) j} \]
as claimed.
\end{proof}

The condition that all terms of the sequence are distinct in Theorem \ref{th:hyp_ok} is mandatory to ensure that no cycle appears. Otherwise, the inequality \eqref{eq:sharp} would be easily falsified, by considering $n=4 \in I(6,2)$ or $n=27 \in I(74,43)$ for instance.

\subsection{Numerical results}
\label{subsec:comp}

We computed the trajectory of all integers $n \leq 10^{9}$. According to the calculations, LBH holds with $C=0$ for $n \leq 10^{9}$ in all cases where $j \neq q$, except for a few integers $n$ given in Table \ref{tab:comp1}. The value $c(n)$ in Table \ref{tab:comp1} denotes the smallest non-negative real such that the lower bound \eqref{eq:min} from LBH holds for $n$ provided $C \geq c(n)$, whatever the number $j$ of iterations. Note that for $n=1$, the case $j=1$ is excluded.

\begin{xrem} One may verify that $c(n)$ exists whenever the trajectory of $n$ contains the value 1. This follows from the fact that $1/2$ is a critical point for the $H$ function.
\end{xrem}

\begin{table}
\begin{center}
\begin{tabular}{|rrrcc|}
  \hline
  n & j & q & q/j & c(n) \\
  \hline
  1 & 3 & 2 & 0.666 & 0.154 \\
  27 & 45 & 33 & 0.733 & 0.472 \\
  31 & 42 & 31 & 0.738 & 0.408 \\
  41 & 44 & 32 & 0.727 & 0.265 \\
  47 & 41 & 30 & 0.731 & 0.195 \\
  54 & 46 & 33 & 0.717 & 0.132 \\
  55 & 46 & 33 & 0.717 & 0.127 \\
  62 & 43 & 31 & 0.720 & 0.058 \\
  63 & 43 & 31 & 0.720 & 0.053 \\
  73 & 48 & 34 & 0.708 & 0.001 \\
  159487 & 35 & 32 & 0.914 & 0.574 \\
  212649 & 37 & 33 & 0.891 & 0.195 \\
  239231 & 34 & 31 & 0.911 & 0.293 \\
  358847 & 33 & 30 & 0.909 & 0.008 \\
  5095423 & 29 & 28 & 0.965 & 0.091 \\
  19638399 & 199 & 140 & 0.703 & 0.034 \\
  21916159 & 40 & 37 & 0.925 & 0.045 \\
  319804831 & 91 & 77 & 0.849 & 0.980 \\
  379027947 & 96 & 80 & 0.833 & 0.774 \\
  426406441 & 93 & 78 & 0.838 & 0.773 \\
  479707247 & 90 & 76 & 0.844 & 0.776 \\
  568541921 & 95 & 79 & 0.831 & 0.575 \\
  598957743 & 103 & 84 & 0.815 & 0.418 \\
  639609663 & 92 & 77 & 0.836 & 0.571 \\
  719560871 & 89 & 75 & 0.842 & 0.571 \\
  758055895 & 97 & 80 & 0.824 & 0.386 \\
  852812883 & 94 & 78 & 0.829 & 0.375 \\
  898436615 & 102 & 83 & 0.813 & 0.226 \\
  959414495 & 91 & 76 & 0.835 & 0.368 \\
 \hline
\end{tabular}
\caption{Integers $n \leq 10^{9}$ such that $c(n) > 0$ and $j,q$ values from which $c(n)$ has been derived. All cases where $j=q$ are omitted.}
\label{tab:comp1}
\end{center}
\end{table}

We found three successive records 0.472, 0.574 and 0.980 for the values of $c(n)$, corresponding to the integers $n=27$, 159487 and 319804831, in that order. These numbers are already known as ``maximum excursion" record-holders for the $3x+1$ problem \cite{Oli10} (see also \S\ref{subsec:TST_ME}). Table \ref{tab:records} gives the other known record-holders for the maximum excursion that are leading to non-zero values of the $c$ function. 

As a result of these calculations, if we assume LBH, then 
\begin{equation}\label{eq:C_min}
C \geq 0.980916600\ldots.
\end{equation}

However, we previously omitted the case $j=q$, which occurs when $n \equiv -1 \pmod{ 2^{j}}$. For $j>1$, we get the positive lower bound
\begin{equation} \label{eq:C_min2}
 c(2^j -1) \geq \frac{-\log(1-2^{-j})}{\log j}
\end{equation}
by considering exactly $j$ iterations. The highest value of the lower bound \eqref{eq:C_min2} is $ \log(4/3)/\log(2) = 0.415\ldots$ obtained when $j=2$. We verified that the equality holds in \eqref{eq:C_min2} for all $j \leq 1000$ except the cases $j=5$ and $j=6$ for which it is necessary to operate more than $j$ iterations of $T$ (see $n=31$ and 63 in Table \ref{tab:comp1}).

One observes, mainly in Table \ref{tab:comp1}, that the integers $n$ for which $c(n) > 0$ tend to form ``clusters" of sequences with very similar lengths and ones-ratios. This is due to a well-known phenomenon of \textit{coalescence} of sequences. For example, the trajectories starting from 27 and 31 are almost identical since $T^{(3)}(27) = 31$.

\begin{table}
\begin{center}
\begin{tabular}{|rrrcc|}
  \hline
  n & j & q & q/j & c(n) \\
  \hline
  1410123943 & 197 & 144 & 0.730 & 0.145 \\
  272025660543 & 109 & 91 & 0.834 & 0.081 \\
  871673828443 & 107 & 91 & 0.850 & 0.327 \\
  3716509988199 & 201 & 155 & 0.771 & 0.426 \\
  9016346070511 & 202 & 155 & 0.767 & 0.113 \\
  1254251874774375 & 227 & 175 & 0.770 & 0.076 \\
  10709980568908647 & 298 & 222 & 0.744 & 0.077 \\
  1980976057694848447 & 399 & 292 & 0.731 & 0.408 \\
 \hline
\end{tabular}
\caption{Known maximum excursion record-holders $n > 10^{9}$ for which $c(n) > 0$.}
\label{tab:records}
\end{center}
\end{table}

\section{Back to the $3x+1$ problem}
\label{sec:syracuse}

Hypothesis \ref{hyp:LBH2} (LBH) asserts that all integers $n \in I(j,q)$ with $j \neq 2q$ are lower bounded by a quasi-exponential function of $j$ whose growth rate depends on the ones-ratio $q/j$. Moreover it implies that the ones-ratio of any given trajectory always converges to 1/2 as the number of iterations tends to $\infty$, thus maximizing the binary entropy function. Proving this property would be sufficient to solve the $3x+1$ problem, as shown by the next lemma.

\begin{lemma} \label{lem:syracuse}
LBH implies that the trajectory of any positive integer leads to the trivial cycle ($3x+1$ problem).
\end{lemma}

\begin{proof}
Suppose that the trajectory of a given positive integer $n$ does not contain the value 1. We may assume, without loss of generality, that $n$ is the smallest term in the trajectory: $3 \leq n \leq T^{(j)}(n)$ for any $j$.

We obtain by Lemma \ref{lem:ratio}
\[n \leq T^{(j)}(n) \leq 2^{-j} (3 + 3^{-1})^{q} \cdot n\]
where $q$ is the number of odd terms among $n, T(n), \ldots, T^{(j-1)}(n)$.  This gives
\[(3 + 3^{-1})^{q} \geq 2^{j}.\]
Therefore the ones-ratio $r = q/j$ has lower bound
\[r \geq r_{0} = \frac{\log 2}{\log(3+3^{-1})} = 0.575 \ldots \]
and $H(r)$ is upper bounded by $H(r_{0}) = 0.983\ldots$. It follows that the right-hand side in \eqref{eq:min} is unbounded as $j$ tends to infinity,  yielding a contradiction with LBH. 
\end{proof}

\section{Dynamic behaviour}
\label{sec:dyn}

\subsection{Total stopping time and maximum excursion}
\label{subsec:TST_ME}

Since Crandall \cite{Cra78}, the dynamic of $T$ is often compared to a multiplicative random walk with a downward drift. Several stochastic models \cite{Bor01,Kon10,Lag92} have been proposed in order to explain the empirical observations concerning the \textit{total stopping time} $\sigma_{\infty}(n)$ and the \textit{maximum excursion} $t(n)$ of a trajectory starting from $n$. Recall that $\sigma_{\infty}(n)$ is the number of iterations until the first occurrence of 1, and $t(n)$ is the highest term of the trajectory. Hypothetically, we set $\sigma_{\infty}(n) = \infty$ if the trajectory of $n$ does not contain the value 1, and $t(n) = \infty$ if it is unbounded. The stochastic models predict that 

\begin{equation} \label{eq:limsup_sigma}
  \limsup_{n \to \infty} \frac{\sigma_{\infty}(n)}{\log n} = \gamma_{RW} \approx 41.677647
\end{equation}
and
\begin{equation} \label{eq:limsup_log_t}
 \limsup_{n \to \infty} \frac{\log t(n)}{\log n} = 2,
\end{equation}

which is consistent with all the empirical data provided independently by Oliveira e Silva and Roosendaal. As reported in \cite{Kon10}, the highest known value of the ratio $\sigma_{\infty}(n)/\log n$ is equal to $36.716\ldots$, due to the finding of a new  record-holder $n \approx 7.21 \times 10^{21}$. The accuracy of \eqref{eq:limsup_log_t} is also discussed in \cite{Kon10,Oli10}, where all known integers $n$ such that $t(n) > n^2$ are given, starting with $n=27$. At the time of writing, the verification process covers all integers $n$ up to $5.76 \times 10^{18}$, thanks to various optimization techniques like the use of search trees for the preselection of congruence classes \cite{Oli10}.

\subsection{Main result}

According to the next theorem, the expected dynamic behaviour of $T$ under iteration, described as above, may be derived from LBH with some refinement to the second order.

\begin{theorem} \label{th:dyn}
Assume LBH. Then there hold the upper bounds
\begin{equation} \label{eq:upper_bound_sigma}
\sigma_{\infty}(n) \leq \gamma_H \log n + O(\log \log n)
\end{equation}
and
\begin{equation} \label{eq:upper_bound_t}
 \log t(n) \leq 2 \log n + O(\log \log n)
\end{equation}
for integers $n \geq 2$, where
\begin{equation} \label{eq:gamma_H}
 \gamma_H = (\log 2 - r_{H} \log 3)^{-1} = 41.677647655\ldots 
\end{equation}
with $r_{H} = 0.609089\ldots$ defined as in Theorem \ref{th:hyp_ok}.
\end{theorem}

\begin{proof}[Proof of \eqref{eq:upper_bound_sigma}]
Let $n \geq 2$. Assuming LBH, we have $\sigma_{\infty}(n) < \infty$, by Lemma \ref{lem:syracuse}. Set $j = \sigma_{\infty}(n)$, so that $T^{(j)}(n) = 1$.

Let $q$ be the number of odd terms among $n, T(n), \ldots, T^{(j-1)}(n)$.

The case $q=0$ is simply stated: $n=2^{j}$ and $j=(\log 2)^{-1} \log n$.

Consider now the case $0 < r=q/j \leq r_H$. Since $n, T(n), \ldots, T^{(j)}(n)$ are distinct, there holds formula \eqref{eq:sharp} in Theorem \ref{th:hyp_ok}, which in turn implies
\[ j \leq \frac{\log n + \frac 16 \log j}{(1-\rho r)\log 2}\]
with
\[(1-\rho r)\log 2 \geq (1-\rho r_H)\log 2 = \gamma_H^{-1}.\]

The case $r > r_H$ gives by assuming LBH
\[ j \leq \frac{\log n + C \log j}{(1-H(r))\log 2}\]
with $C \geq 0$ a constant, and
\[(1-H(r))\log 2 > (1-H(r_H))\log 2 = \gamma_H^{-1}.\]

Thus, in all cases, we obtain the upper bound
\[j \leq \gamma_H ( \log n + C \log j )\]
since $C > 1/6$, according to \eqref{eq:C_min}. It yields $ j = O(\log n)$, and we finally get
\[ \sigma_{\infty}(n) \leq \gamma_H \log n + O(\log \log n)\]
as claimed.
\end{proof}

\begin{proof}[Proof of \eqref{eq:upper_bound_t}]
Let $n \geq 3$ be an odd integer. Assuming LBH, we have $t(n) < \infty$, by Lemma \ref{lem:syracuse}. Let $j \geq 1$ such that $T^{(j)}(n)=t(n)$. We may suppose, without loss of generality, that $n$ is the smallest term among $n, T(n), \ldots, T^{(j-1)}(n)$.

Using Lemma \ref{lem:ratio}, we get
\[ \frac{t(n)}{n} \leq \frac{3^q}{2^j}  \left( 1 + \frac{1}{3 n} \right)^{q}\]
where $q$ is the number of odd terms in the iterated sequence going from $n$ to $T^{(j-1)}(n)$. Then we divide by $n$ and apply LBH:
\[ \frac{t(n)}{n^{2}} \leq j^{C} \cdot 2^{(\rho r + H(r) - 2)j} \cdot \left( 1 + \frac{1}{3 n} \right)^{q}\]
with $\rho = \log_{2} 3$ and $r = q/j$. Now one verifies that $H(x) + \rho x \leq 2$ for any $x$ where the equality holds if and only if $x=3/4$. It follows that 
\[ \frac{t(n)}{n^{2}} \leq j^{C} \cdot \left( 1 + \frac{1}{3 n} \right)^{q}\]
and, taking the logarithm,
\[ \log t(n) \leq 2 \log n + C \log j + \frac{q}{3n}\]
with the upper bound $\log(1+x) \leq x$. Since $q \leq j \leq \sigma_{\infty}(n)$, the inequality \eqref{eq:upper_bound_sigma} gives the claimed result
\[ \log t(n) \leq 2 \log n + O(\log \log n). \]
This completes the proof of Theorem \ref{th:dyn}.
\end{proof}
The above proof suggests that a maximum excursion record is more likely to occur when the ones-ratio of the sequence that goes from $n$ to $t(n)$ is approximately $3/4$. This prediction is in good agreement with the empirical data in Table \ref{tab:records}, mostly for long sequences.

Interestingly, the fact that $\gamma_H$ does not depend on the value of the constant $C$ in LBH further strengthens its relevance.

\subsection{From the random walk model to entropy}
\label{subsec:RW_model}

One may ask whether the constants $\gamma_{RW}$ and $\gamma_H$ are identical. Though they are seemingly the same, their respective definitions differ. Recall that $\gamma_{RW}$ originated in a model described in \cite{Lag92}, namely the \textit{random walk model}. 
 
 For each integer $n \geq 1$, Lagarias \& al consider a sequence of i.i.d. random variables $X(n,k)$, $k \geq 1$,  taking their values in the discrete set $\left\lbrace  \log2, \log \frac 23 \right\rbrace$ with the same probability $\frac 12$. Starting from $\log n$, each random variable represents a single step towards $-\infty$ within some additive random walk on a logarithmic scale. 
 
 Using Chernoff's bound from the theory of large deviations, it was stated that almost surely
 \[ \limsup_{n \to \infty} \frac{\min_{k \geq 1} \left\lbrace k : \log n - \sum_{i=1}^{k} X(n,i) \leq 0 \right\rbrace }{\log n} = \gamma_{RW}\]
 where $\gamma_{RW}$ is the unique solution with $\gamma_{RW} > \left( \frac 12 \log \frac 43 \right)^{-1}$ of the equation
\begin{equation} \label{eq:gamma_rw}
 \gamma_{RW} \cdot  g\left( \frac{1}{\gamma_{RW}} \right) = 1.
\end{equation}
The rate function $g$ above is the Legendre transform 
\begin{equation} \label{eq:rate_rw}
g(a) = \sup_{\theta \in \mathbb{R}} \left( a \theta - \log  M_{RW}(\theta)\right)
\end{equation}
defined for $\log \frac 23 < a < \log 2$, with the moment generating function 
\[M_{RW}(\theta) = \frac 12 \left(  2^{\theta} + \left(\frac 23 \right)^{\theta} \right).\]

We give a more simple expression of the rate function in the next lemma\footnote{Lemma \ref{lem:rw} relates the rate function to the binary entropy function. It can be derived directly by using another form of Chernoff's bound for the simple case of a Bernoulli distribution and based on the notion of relative entropy (see, e.g., \cite{Cov91} for this formulation). }.

\begin{lemma} \label{lem:rw}
Let $0 < r < 1$. Then
\begin{equation} \label{eq:rate_entropy}
 g(\log 2 - r \log 3) = \left( 1 - H(r) \right) \log 2
\end{equation}
 where $g$ is defined as in \eqref{eq:rate_rw} and $H$ is the binary entropy function.
\end{lemma}

\begin{proof}
Set $a = \log 2 - r \log 3$ and suppose that $\theta^{*}$ verify
\[g(a) = a \theta^{*} - \log  M_{RW} \left(\theta^{*}\right).\]
Writing the condition
\[ \frac{d}{d \theta} \left( a \theta - \log  M_{RW}(\theta) \right) \Bigr|_{\theta=\theta^{*}}
   = 0 \]
leads to the relation
\[ a - \log 2 + \frac{\log 3}{3^{\theta^{*}} + 1} = 0, \]
which simplifies to
\[ r \left( 3^{\theta^{*}} +1 \right) = 1.\]
We obtain after calculation
\[ g(a) = r \log r + (1-r) \log(1-r) + \log 2. \qedhere\] 
\end{proof}

\begin{corollary}
Let $\gamma_{RW}$ be defined by \eqref{eq:gamma_rw}, and let $\gamma_H = (\log 2 - r_{H} \log 3)^{-1}$ with $r_{H} = 0.609\ldots$ such that
$H(r_{H}) \log 2 = r_{H} \log 3.$ Then
\begin{equation} \label{eq:gammas}
 \gamma_{RW} = \gamma_H. 
\end{equation}
\end{corollary}

\begin{proof}
Setting $\gamma_{RW} = \left( \log 2 - r_{RW} \log 3 \right) ^{-1}$, we get from \eqref{eq:gamma_rw} and Lemma \ref{lem:rw} the equation
\[ H(r_{RW}) \log 2 = r_{RW} \log 3\]
with $r_{RW} > \frac 12$. Thus, $r_{RW} = r_H$, since $r_H$ satisfies the same equation that has a unique positive solution (see Theorem \ref{th:hyp_ok}). The expected result \eqref{eq:gammas} follows.
\end{proof}

The authors of \cite{Lag92} also consider another stochastic model, namely the \textit{branching process model}. A tree is build from a family of Bernoulli processes that imitate the backward iterations of the $T$ function on a logarithmic scale. Applying a theorem of Biggins based on a Chernoff's bound for Galton-Watson processes, this model predicts that the left-hand side in \eqref{eq:limsup_sigma} equals a constant $\gamma_{BP}$ . Then it is shown that $\gamma_{BP} = \gamma_{RW}$, which is quite satisfying.

The identities $\gamma_{BP} = \gamma_{RW} = \gamma_H$  suggest that LBH is in full agreement with the predictions of both the random walk and the branching process models. As a possible explanation, one might consider that the heuristic reasonings leading to all of them are somehow related.

\section{A random model of uniform distribution}
\label{sec:model}

The assumptions made in \S\ref{subsec:heuristic} give no indication on the value of the constant $C$ in LBH. To this end, let us consider a random model\footnote{This model is more simple than the random walk model \cite{Lag92}. We point out that, in our model, the random representation of $I_{0}(j,q)$ has exactly $\binom{j}{q}$ elements, whereas, in the random walk model, the number of sequences of length $j$, starting from $\log n$ with $n \leq 2^j$, and having $q$ terms considered as ``odd" is a Gaussian random variable with mean $\binom{j}{q}$. Yet we expect that those models lead to similar predictions regarding LBH.}  where the elements $n \in I_{0}(j,q)$ are represented by a set of independent random variables $\left\lbrace X_{j,q,i} : i=1,2, \ldots, \binom{j}{q} \right\rbrace $ having a continuous uniform distribution on the interval $[0, 2^j ]$.

Let $P(j,q)$ denote the probability that 
\begin{equation}\label{eq:random_var}
 X_{j,q,i} < j^{-C} \cdot 2^{(1-H(r)) j}
\end{equation}
where $i,j,q,r$ are taken such that $0 \leq q \leq j$ ($j \neq 0$), $1 \leq i \leq \binom{j}{q}$ and $r = q/j$. The value of $P(j,q)$ obviously does not depend on $i$:
\[P(j,q) = j^{-C} \cdot 2^{-H(r) j} .\]
Let us introduce the infinite sum of probabilities
\[ S(C) = \sum_{j=1}^{\infty} \sum_{q=0}^{j} \binom{j}{q} P(j,q), \]
which estimates the number of times the inequality \eqref{eq:random_var} is satisfied over all admissible values of $i,j,q$.
By the Stirling formula, we have
\begin{equation}\label{eq:stirling}
\binom{j}{q} \sim \frac{2^{H(r) j}}{\sqrt{2 \pi r(1-r)j}} .
\end{equation}
for $\varepsilon \leq r \leq 1 - \varepsilon$, with $\varepsilon >0$ fixed. Then we get the approximations
\[ \binom{j}{q} P(j,q) \sim \frac{j^{-\frac 12 - C}}{\sqrt{2 \pi r(1-r)}} \]
and
\[ \sum_{q=0}^{j} \binom{j}{q} P(j,q) \sim \frac{j^{\frac 12 -C}}{\sqrt{2\pi}}  \int_{0}^{1}\frac{dx}{\sqrt{x(1-x)}} = \sqrt{\frac{\pi}{2}} \cdot j^{\frac 12 -C}. \]
One may verify (e.g., by developing the Stirling series to the next order) that the latter approximation is still valid when summing on $q$.

We infer that $S(C) < \infty$ if and only if $C > 3/2$, by considering the conditional convergence of the Riemann Zeta function on the real line. Thus, almost surely, inequality \eqref{eq:random_var} occurs at most finitely many times when $C > 3/2$, as a consequence of the Borel-Cantelli lemma.

This simple model suggests that LBH is likely to be true for any $C>3/2$ and all $j$ sufficiently large. Nevertheless, there is no randomness in the sets $I_{0}(j,q)$, and the previous estimation of the plausible values of $C$ may be flawed for various reasons\footnote{See also \cite[p. 153]{Kon10} for a similar discussion on the random walk model.}:
\begin{enumerate}[\upshape (i)]
\item The elements of $I_{0}(j,q)$ have a lower bound given by Lemma \ref{lem:hyp_ok0} when $r \leq 1/2$, and all constant $C \geq 0$ is admissible in that case.
\item The values $\min I_{0}(j,q)$ and $\min I_{0}(j+1,q)$ are often equal, and thus, correlated. For example, 
\[\min I_{0}(50,30) = \min I_{0}(51,30) = 103.\]
\item The values $\min I_{0}(j,q)$ and $\min I_{0}(j,q+1)$ are interdependent when leading to coalescent sequences after a few iterations.
\end{enumerate}

Therefore, it remains plausible that LBH holds for a constant $C$ lower than $3/2$. Yet, the exact value of $C$ does not matter in most cases, as the first order in the lower bound \eqref{eq:min} is the exponential term.

\section{A particular case}

\subsection{Effective lower bound}

On a theoretical level, we know very little about the smallest element of the set $I_{0}(j,q)$ in all cases where $\log_3 2 < q/j < 1$, which relates to sequences that tend to grow and have at least one even term. Here we briefly investigate the most simple of those cases, that is $q=j-1$, for which $\# I_{0}(j,j-1) = j$.

\begin{lemma}
Let $j \geq 2$. Then $I_{0}(j,j-1) = \left\lbrace n_{j,k} \right\rbrace_{k=0}^{j-1} $ with \[n_{j,0} = 2^{j} -2 \quad \text{and} \quad n_{j,k} = \left( \frac 23 \right) ^{k} \left( b_{k}(j-k) \cdot 2^{j-k} -1 \right) -1\] for $1 \leq k \leq j-1$, where $b_{k}(l) = 2^{-l} \mod 3^k$. Moreover, the set $I_{0}(j,j-1)$ is bounded from below by
\begin{equation} \label{eq:lower_bound_njk}
 2^{j/(1+\rho)} - 2
\end{equation}
with $\rho = \log_2 3$.
\end{lemma}

\begin{proof}
It is straightforward to state that the first $j-1$ iterates of $n_{j,k}$ are odd integers, except $T^{(k)}\left( n_{j,k} \right)$. We have indeed 
\[T^{(k)}\left( n_{j,k} \right) = b_{k}(j-k) \cdot 2^{j-k} - 2\] for $1 \leq k \leq j-1$.
Thus, it suffices to observe that $1 \leq n_{j,k} \leq 2^{j}$ for all $k$ to conclude that the $j$ elements of $I_{0}(j,j-1)$ are the $n_{j,k}$ integers.

Now we are left to prove \eqref{eq:lower_bound_njk}. On the one hand, we can write 
\begin{equation} \label{eq:lb_njk_1}
 n_{j,k} \geq \left( \frac 23 \right) ^{k} \left( 2^{j-k} -1 \right) -1 \geq \frac{2^j}{3^k}-2 = 2^{j-\rho k} - 2 
\end{equation}
for $1 \leq k \leq j-1$, by using $b_{k}(l) \geq 1$. On the other, 
\begin{equation} \label{eq:lb_njk_2}
 n_{j,k} = 2^k \left( \frac{b_{k}(j-k) \cdot 2^{j-k} -1}{3^k}  \right) -1 \geq 2^k - 1 
\end{equation}
which follows from the fact that $(b_{k}(l) \cdot 2^{l} -1)$ is a non-zero multiple of $3^k$ for any $l \geq 1$. Putting \eqref{eq:lb_njk_1} and \eqref{eq:lb_njk_2} together yields
\[ n_{j,k} \geq 2^{\max ( k, j-\rho k)} - 2.\]
The lower bound 
\[ \min_{1 \leq k \leq j-1} \max ( k, j-\rho k) \geq \frac{j}{1 + \rho}\] 
completes the proof.
\end{proof}

The effective lower bound \eqref{eq:lower_bound_njk} has an exponential growth, but it is quite weak compared to LBH which asserts that
\begin{equation} \label{eq:lower_bound_njk_hyp}
 \min I_{0}(j,j-1) \geq j^{-C} \cdot 2^{\left( 1-H(1-1/j) \right) j} \sim e^{-1} \cdot j^{-(C+1)} \cdot 2^j .
\end{equation}
A simple calculation shows that the assumption \eqref{eq:lower_bound_njk_hyp} on the $n_{j,k}$ integers leads to the roughly equivalent lower bound
\begin{equation} \label{eq:conj_bkl}
 b_{k}(l) \geq \frac{3^{k}}{ e \cdot (l+k)^D} \quad \text{ for all $k,l$ positive integers, }
\end{equation}
 where $D \approx C+1$ is a constant. To our knowledge, proving \eqref{eq:conj_bkl} is a non-trivial problem in number theory.

As the multiplicative group $(\mathbb{Z}/3^k\mathbb{Z})^{*}$ is cyclic of order $2 \cdot 3^{k-1}$, the $b_{k}$ functions are periodic with period $2 \cdot 3^{k-1}$, and there holds
\[b_{k}(2 \cdot 3^{k-1}) = 1.\]
The inverse functions $b_{k}^{-1}$ are the discrete logarithms in base 1/2, modulo $3^k$. Thus, the hypothetical lower bound  \eqref{eq:conj_bkl} may be related to the distribution of discrete logarithms \cite{Gib12}, which is expected to be uniform.

\subsection{Further numerical results}
In order to test numerically LBH in that particular case, we checked the lower bound \eqref{eq:min} by setting $q=j-1$, $n=n_{j,k}$  and $C=0$ for all $0 \leq k < j \leq 10000$. When it was falsified, which occurred 3741 times, then we computed $c(n)$, where the $c$ function is defined as in \S\ref{subsec:comp}. Here we only give the three highest values found:
\[ c\left( n_{85,56} \right) = 0.865\ldots,\]
\[ c\left( n_{2858,1270} \right) = 0.817\ldots,\]
\[ c\left( n_{5461,488} \right) = 0.813\ldots.\]
The three above $n_{j,k}$ integers have 22, 854 and 1637 decimal digits, in that order. Let us mention that the corresponding $c(n_{j,k})$ values have been obtained after exactly $j$ iterations. These numerical results do not improve the bound \eqref{eq:C_min}, supporting the idea that LBH may hold with $C$ near 1.


\section*{Acknowledgements}
I wish to acknowledge for the fruitful correspondence with Claude Terracol that originated this research. I also thank Christian Boyer, Nik Lygeros and Cl\'ement Narteau for their positive influences, and the referee for his comments and suggestions.



\begin{thebibliography}{99}

\normalsize
\baselineskip=17pt

\bibitem{Bor01}
K. A. Borovkov, D. Pfeifer,
{\em Estimates for the Syracuse problem via a probabilistic model},
Theory Probab. Appl. 45 (2001), 300--310.

\bibitem{Cov91}
T. M. Cover, J. A. Thomas, 
{\em Elements of Information Theory}, John Wiley, New York, 1991.

\bibitem{Cra78}
R. E. Crandall,
{\em On the $3x+1$ problem}, 
Math. Comp. 32 (1978), 1281--1292.

\bibitem{Eli93}
S. Eliahou, 
{\em The $3x + 1$ problem: new lower bounds on nontrivial cycle lengths},
Discrete Math. 118 (1993), 45--56.

\bibitem{Eve77}
C. J. Everett,
{\em Iteration of the number theoretic function $f(2n)=n, f(2n+1)=3n+2$},
Adv. Math. 25 (1977), 42--45.

\bibitem{Gib12}
D. J. Gibson,
{\em Discrete logarithms and their equidistribution},
Unif. Distrib. Theory 7 (2012), 147--154.

\bibitem{Kon10}
A. V. Kontorovich, J. C. Lagarias,  
{\em Stochastic models for the $3x+ 1$ and $5x+ 1$ problems and related problems}, published in \cite{Lag10}, 131--188.

\bibitem{Lag85}
J. C. Lagarias,
{\em The $3x+1$ problem and its generalizations},
Amer. Math. Monthly 92 (1985), 3--23.

\bibitem{Lag92}
J. C. Lagarias, A. Weiss,
{\em The $3x + 1$ problem: two stochastic models},
Ann. Appl. Probab. 2 (1992), 229--261.

\bibitem{Lag10}
J. C. Lagarias, editor,
{\em The ultimate challenge: The 3x+ 1 problem}, Amer. Math. Soc., 2010.

\bibitem{Oli10}
T. Oliveira e Silva,
{\em Empirical verification of the $3x + 1$ and related conjectures}, published in \cite{Lag10}, 189--207.

\bibitem{Sin03}
Y. G. Sinai, 
{\em Statistical $(3x+1)$ problem},
Commun. Pure Appl. Math. 56 (2003), 1016--1028.

\bibitem{Tao11}
T. Tao,
{\em The Collatz conjecture, Littlewood-Offord theory, and powers of 2 and 3}, blog post published August 25, 2011 at \\
\url{http://terrytao.wordpress.com/2011/08/25} .

\bibitem{Ter76}
R. Terras,
{\em A stopping time problem on the positive integers},
Acta Arith. 30 (1976), 241--252.

\bibitem{Wir98}
G. J. Wirsching, 
{\em The Dynamical System Generated by the $3n + 1$ Function}, Lecture Notes in Math. 1681, Springer, 1998.

\end{thebibliography}
\end{document}